\documentclass{proc-l}

\newcommand{\calC}{{\mathcal{C}}}
\newcommand{\calD}{{\mathcal{D}}}
\newcommand{\calE}{{\mathcal{E}}}
\newcommand{\calF}{{\mathcal{F}}}

\newcommand{\calP}{{\mathcal{P}}}

\newcommand{\N}{{\mathbb{N}}} 

\newtheorem{theorem}{Theorem}[section]

\newtheorem{proposition}[theorem]{Proposition}%

\theoremstyle{definition}
\newtheorem{definition}[theorem]{Definition}

\theoremstyle{remark}

\numberwithin{equation}{section}

\begin{document}

\title{The pseudometric topology induced by upper asymptotic density}


\author{Jonathan M. Keith}
\address{School of Mathematics, Monash University, Wellington Rd, Clayton VIC 3800, Australia}
\email{jonathan.keith@monash.edu}

\subjclass[2010]{Primary 11B05; Secondary 54E50}

\date{18 May 2024}

\dedicatory{}

\commby{Nageswari Shanmugalingam}

\begin{abstract}
Upper asymptotic density induces a pseudometric on the power set of the natural numbers, with respect to which $\calP(\N)$ is complete. The collection $\calD$ of sets with asymptotic density is closed in this pseudometric, and closed subsets of $\calD$ are characterised by a generalisation of an additivity property (AP0). 
\end{abstract}

\maketitle

\section{Introduction}

The {\em asymptotic density} of a set $A \subseteq \N$ is defined to be
\[
\nu(A) := \lim_{n \rightarrow \infty} \frac{\lvert A \cap n \rvert}{n},
\]
if that limit exists. (The $n$ in the numerator represents the set $\{ 0, 1, \ldots, n-1 \}$, a common shorthand in the literature on asymptotic density, as for example in~\cite{blass2001}.) The collection of sets with asymptotic density is denoted~$\calD$. If the limit in the above definition is replaced by the limit inferior or limit superior, one obtains the {\em lower asymptotic density} $\nu^-(A)$ and {\em upper asymptotic density} $\nu^+(A)$, respectively. Both $\nu^-$ and $\nu^+$ extend $\nu$ to $\calP(\N)$. These set functions are important in several mathematical disciplines, including number theory (see, for example, their role in numerous theorems and proofs in~\cite{nathanson2000}). Key papers exploring the properties of asymptotic density (and its extensions) include~\cite{blass2001, blumlinger1996, buck1946, buck1953, kadane1995, kunisada2017, kunisada2019, kunisada2022, letavaj2015, maharam1976, mekler1984, polya1929, sleziak2008, vandouwen1992}. 

Asymptotic density is additive in the sense that for disjoint $A, B \in \calD$, the set $A \cup B$ is in $\calD$ with $\nu(A \cup B) = \nu(A) + \nu(B)$. Lower asymptotic density is superadditive ($\nu^-(A \cup B) \geq \nu^-(A) + \nu^-(B)$), whereas upper asymptotic density is subadditive ($\nu^+(A \cup B) \leq \nu^+(A) + \nu^+(B)$). Much of the literature cited above is concerned with additive extensions of $\nu$ to $\calP(\N)$. A highlight of this literature is the identification of necessary and sufficient conditions for the $L^p$ spaces induced by such an additive extension to be complete~\cite{blass2001, kunisada2017}.

This paper investigates the pseudometric topology induced on $\calD$ by the upper asymptotic density, via the pseudometric
\[
d(A, B) := \nu^+(A \triangle B),
\]
defined for all $A, B \in \calP(\N)$. Note that one does not obtain a pseudometric on $\calD$ if $\nu^+$ is replaced by $\nu$ in this definition, because $A \triangle B$ is not necessarily an element of $\calD$ for arbitrary $A, B \in \calD$. A pseudometric is obtained if $\nu^+$ is replaced by an additive extension $\overline{\nu}$ of $\nu$ to $\calP(\N)$, and it is possible (assuming the ultrafilter principle) to choose $\overline{\nu}$ so that $\calP(\N)$ is complete (this follows from Theorem~1 of~\cite{blass2001} and Theorem~1 of~\cite{gangopadhyay1999}). However, it is not known whether $\overline{\nu}$ can be chosen so that $\calD$ is also closed in the induced topology.
With $d$ defined as above, the following theorem is proved in Section~\ref{complete_section}.

\begin{theorem} \label{calF_complete}
$(\calP(\N), d)$ is complete and $\calD$ is closed.
\end{theorem}

Moreover, the closed subsets of $\calD$ can be characterised in terms of a property that generalises AP0, as defined by Mekler~\cite{mekler1984} and others. In the context of asymptotic density, AP0 can be stated as follows: for every increasing sequence $( A_i )_{i\in \N}$ in $\calD$, there is $A \in \calD$ such that $\nu(A_i \setminus A) = 0$ for all $i$, and $\nu(A) = \lim_{i \rightarrow \infty} \nu(A_i)$. An additive extension $\overline{\nu}$ of $\nu$ to $\calP(\N)$ has AP0 if the above statement holds with $\calP(\N)$ in place of $\calD$ and $\overline{\nu}$ in place of $\nu$. One reason AP0 is an important property is that it characterises those additive extensions of $\nu$ for which $\calP(\N)$ is complete with respect to the pseudometric $d_{\overline{\nu}}(A,B) := \overline{\nu}(A \triangle B)$ (see Theorem~1 of~\cite{gangopadhyay1999}). 

The generalisation of AP0 can be expressed in terms of quotients. Let $[A]$ be the equivalence class of $A \in \calP(\N)$, that is
\[
[A] := \{ B \in \calP(\N) : d(A,B)=0 \}.
\]
Define
\[
\calP(\N) / d := \{ [A] : A \in \calP(\N) \} \quad \mbox{ and } \quad \calD / d := \{ [A] : A \in \calD \}.
\]
Then $d([A],[B]) := d(A,B)$ defines a metric on $\calP(\N) / d$ or $\calD / d$. (In fact $\calP(\N) / d$ is also a Boolean quotient, since the null sets of $\nu^+$ form an ideal in $\calP(\N)$.) The following theorem is proved in Section~\ref{closed_section}.

\begin{theorem} \label{closed_limsup_liminf}
A set $\calE \subseteq \calD / d$ is closed if and only if the following condition holds: for every sequence $( A_i )_{i\in \N}$ in $\calE$ such that the nets 
\[
\{ \nu^+(\vee_{k = i}^j A_k) : i, j \in \N, i \leq j \}
\quad \mbox{ and } \quad
\{ \nu^-(\vee_{k = i}^j A_k) : i, j \in \N, i \leq j \}
\]
converge to a common limit $L$, the net $\{ \vee_{k = i}^j A_k : i, j \in \N, i \leq j \}$ converges to some $A \in \calE$ with $\nu^+(A) = L$.
\end{theorem}

Here the directed set $\{ (i,j) : i,j \in \N, i \leq j \}$ has the product order: 
\[
(i,j) \leq (k,l) \iff i \leq k \mbox{ and } j \leq l.
\]
The operator $\vee$ represents the least upper bound in the partial order on $\calP(\N) / d$ given by 
\[
[A] \leq [B] \iff \nu^+(A \setminus B) = 0.
\]
The functions $\nu^+$ and $\nu^-$ are given by $\nu^+([A]) := \nu^+(A)$ and $\nu^-([A]) := \nu^-(A)$ respectively. These functions can be shown to be well defined. The property described in Theorem~\ref{closed_limsup_liminf} is equivalent to AP0 in the case $\calE = \calD / d$ and $( A_i )_{i\in \N}$ is increasing. 

The closed subsets of $\calD$ are then the inverse images under the quotient map $A \mapsto [A]$ of closed subsets of $\calD / d$. 

\section{Completeness of $(\calD, d)$} \label{complete_section}

Theorem~\ref{calF_complete} can be proved by a diagonalisation argument. Let $(B_k)_{k=1}^{\infty}$ be a Cauchy sequence in $(\calP(\N), d)$, and define $\epsilon_k := \nu^+(B_k \triangle B_{k+1})$ for each $k$, so that $\epsilon_k \rightarrow 0$. Without real loss of generality, it may be supposed $\epsilon_k > 0$ for all $k$, since either $(B_k)_{k=1}^{\infty}$ has a subsequence with this property, or if no such subsequence exists, the sequence trivially converges. Define a putative limit for $(B_k)_{k=1}^{\infty}$, of the form:
\[
B := \cup_{k=1}^{\infty} \left( B_k \cap (n_{k - 1}, n_k] \right),
\]
where $n_0 = -1$ and $n_k > n_{k - 1}$ for each $k$. Let $n_1$ be the smallest positive integer such that $\nu_n(B_1 \triangle B_2) < \nu^+(B_1 \triangle B_2) + \epsilon_1$ for all $n \geq n_1$, where 
\[
\nu_n(A) :=  \frac{\lvert A \cap n \rvert}{n}
\]
for all $A \in \calP(\N)$ and $n \in \N$. Iteratively define $n_k$ (for $k \geq 2$) to be the smallest positive integer such that:
\begin{enumerate}
\item $\nu_n(B_j \triangle B_{k+1}) < \nu^+(B_j \triangle B_{k+1}) + \epsilon_k$ for all $n \geq n_k$ and $j \leq k$, and
\item $\frac{n_{k - 1}}{n_k} < \epsilon_k$.
\end{enumerate}

Observe that 
\begin{equation}
| \nu_{n_k}(B_j \triangle B) - \nu_{n_k}(B_j \triangle B_k) | <  \epsilon_k
\end{equation} 
for $1 \leq j \leq k$. This is trivially true in the case $j = k = 1$, since $\nu_{n_1}(B_1 \triangle B) = \nu_{n_1}(B_1 \triangle B_1) = 0$. In the case $1 < j = k$, note 
\[
\nu_{n_k}(B_k \triangle B) = \frac{1}{n_k} n_{k-1} \nu_{n_{k-1}}(B_k \triangle B) < \epsilon_k
\]
and $\nu_{n_k}(B_k \triangle B_k) = 0$. In the case $j < k$, first observe that
\begin{align}
\nu_n(B_j \triangle B) & =  \frac{1}{n} \left( n_k \nu_{n_k}(B_j \triangle B) + n \nu_n(B_j \triangle B_{k+1}) - n_k \nu_{n_k}(B_j \triangle B_{k+1}) \right) \nonumber \\
& =  \frac{n_k}{n} \left( \nu_{n_k}(B_j \triangle B) - \nu_{n_k}(B_j \triangle B_{k+1}) \right) + \nu_n(B_j \triangle B_{k+1}),
\end{align}
for $1 \leq j \leq k$ and $n_k \leq n \leq n_{k + 1}$. Now put $n = n_{k+1}$ in (2.2) and note $| \nu_{n_k}(B_j \triangle B) - \nu_{n_k}(B_j \triangle B_{k+1}) | \leq 1$ and $n_k / n_{k + 1} < \epsilon_{k+1}$, thus obtaining 
\[
| \nu_{n_{k+1}}(B_j \triangle B) - \nu_{n_{k+1}}(B_j \triangle B_{k+1}) | <  \epsilon_{k+1}
\]
for $1 \leq j \leq k$. 

Another consequence of (2.2) is that $\nu_n(B_j \triangle B)  \leq \nu^+(B_j \triangle B_{k+1}) + 4 \epsilon_k$. To see this in the case $\nu_{n_k}(B_j \triangle B) - \nu_{n_k}(B_j \triangle B_{k+1}) \leq 0$, observe that (2.2) gives 
\[
\nu_n(B_j \triangle B) \leq \nu_n(B_j \triangle B_{k+1}) < \nu^+(B_j \triangle B_{k+1}) + \epsilon_k
\]
using the first criterion for selecting $n_k$. If $\nu_{n_k}(B_j \triangle B) - \nu_{n_k}(B_j \triangle B_{k+1}) > 0$, then (2.2) gives
\begin{align}
\nu_n(B_j \triangle B) & \leq  \left( \nu_{n_k}(B_j \triangle B) - \nu_{n_k}(B_j \triangle B_{k+1}) \right) + \nu_n(B_j \triangle B_{k+1}) \nonumber \\
& \leq  \left( \nu_{n_k}(B_j \triangle B) - \nu_{n_k}(B_j \triangle B_k) \right) + \nu_{n_k}(B_k \triangle B_{k + 1}) + \nu_n(B_j \triangle B_{k+1}) \nonumber \\
& <  \epsilon_k + \left( \nu^+(B_k \triangle B_{k+1}) + \epsilon_k \right) + \left( \nu^+(B_j \triangle B_{k+1}) + \epsilon_k  \right) \nonumber \\
& =  \nu^+(B_j \triangle B_{k+1}) + 4 \epsilon_k, \nonumber 
\end{align}
where the second line uses the triangle inequality $\nu_{n_k}(B_j \triangle B_k) \leq \nu_{n_k}(B_j \triangle B_{k+1}) + \nu_{n_k}(B_k \triangle B_{k + 1})$ and the third line uses (2.1) and the first criterion for selecting $n_k$. It follows that
\[
\sup _{n_k \leq n \leq n_{k+1}} \nu_n(B_j \triangle B) \leq \nu^+(B_j \triangle B_{k+1}) + 4 \epsilon_k.
\]
Taking the limit superior as $k \rightarrow \infty$ yields 
\[
\nu^+(B_j \triangle B) \leq \limsup_{k \rightarrow \infty} \nu^+(B_j \triangle B_{k+1})
\]
and thus
\[
\limsup_{j \rightarrow \infty} \nu^+(B_j \triangle B) \leq \limsup_{j \rightarrow \infty} \limsup_{k \rightarrow \infty} \nu^+(B_j \triangle B_{k+1}) = 0.
\]
Hence $d(B_j, B) \rightarrow 0$ as $j \rightarrow \infty$, and $(\calP(\N), d)$ is complete.

Since $\nu^+$ and $\nu^-$ are continuous, their coincidence set is closed (see Proposition~\ref{uniform_continuous} below and Proposition~6.9 of~\cite{james1987}). That is, $\calD$ is closed.

\section{Closed subsets of $\calD$} \label{closed_section}

A more general result than Theorem~\ref{closed_limsup_liminf} is proved in this section. Some new terminology and notation will facilitate the exposition. In what follows, $I_A$ represents the indicator function of a set $A$. A {\em field} is a family of sets containing the empty set $\emptyset$ and closed under complements and finite unions.

\begin{definition} \label{sub_and_super}
Consider a set $X$, a field $\calF \subseteq \calP(X)$, and set functions 
\[
\mu^+ : \calF \rightarrow [0,~1] \quad \mbox{ and } \quad \mu^-~:~\calF~\rightarrow~[0,~1]
\]
with $\mu^+(\emptyset) = \mu^-(\emptyset) = 0$ and $\mu^+(X) = \mu^-(X) = 1$.
\begin{enumerate}
\item $\mu^+$ is {\em subadditive} if 
\[
I_A \leq I_{B} + I_{C} \implies \mu^+(A) \leq \mu^+(B) + \mu^+(C)
\]
for $A, B, C \in \calF$.

\item $\mu^-$ is {\em superadditive} if 
\[
I_A \geq I_{B} + I_{C} \implies \mu^-(A) \geq \mu^-(B) + \mu^-(C)
\]
for $A, B, C \in \calF$.

\item $\mu^-$ is {\em co-subadditive with respect to $\mu^+$} if 
\[
I_A \leq I_{B} + I_{C} \implies \mu^-(A) \leq \mu^-(B) + \mu^+(C)
\]
for $A, B, C \in \calF$.

\item $\mu^+$ is {\em co-superadditive with respect to $\mu^-$} if 
\[
I_A \geq I_{B} + I_{C} \implies \mu^+(A) \geq \mu^+(B) + \mu^-(C)
\]
for $A, B, C \in \calF$.

\item The pair $(\mu^+, \mu^-)$ is {\em conjugate} if all four of the above properties hold.
\end{enumerate}
\end{definition}

By these definitions, subadditive and superadditive set functions are necessarily monotone, that is, 
\[
A \subseteq B \implies \mu^+(A) \leq \mu^+(B)
\]
for $A, B \in \calF$, and similarly for $\mu^-$.  

It will also be useful to say that a set function $\mu^+ : \calF \rightarrow [0, 1]$ is {\em countably subadditive} if 
\[
I_A \leq \sum_{k=1}^{\infty} I_{A_k} \implies \mu^+(A) \leq \sum_{k=1}^{\infty} \mu^+(A_k)
\]
for any $A, A_1, A_2, \ldots \in \calF$. 

Using the properties of the $\limsup$ and $\liminf$ operators, it is straightforward to verify that lower and upper asymptotic density form a conjugate pair that agree on $\calD$ (and only on $\calD$, where they also agree with $\nu$).

\begin{definition} \label{mu_plus_pseudometric}
Consider a set $X$, a field $\calF \subseteq \calP(X)$, and a subadditive function $\mu : \calF \rightarrow [0, 1]$. Define
\[
d_{\mu}(A, B) := \mu(A \triangle B)
\]
for all $A, B \in \calF$.
\end{definition}

It is straightforward to show that $d_{\mu}$ satisfies the axioms of a pseudometric. In particular, the triangle inequality follows from the subadditivity of $\mu$. 

The key new theorem in this section can now be stated.

\begin{theorem} \label{complete_limsup_liminf}
Consider a set $X$, a field $\calF \subseteq \calP(X)$, an arbitrary set $\calE \subseteq \calF$, and conjugate set functions $(\mu^+, \mu^-)$ defined on $\calF$ and agreeing on $\calE$. The pseudometric space $(\calE, d_{\mu^+})$ is complete if and only if the following condition holds: for every sequence $( A_i )_{i \in \N}$ in $\calE$ such that the nets 
\[
\{ \mu^+(\cup_{k = i}^j A_k) : i, j \in \N, i \leq j \}
\mbox{ and }
\{ \mu^-(\cup_{k = i}^j A_k) : i, j \in \N, i \leq j \}
\]
converge to a common limit $L$, the net $\{ \cup_{k = i}^j A_k : i, j \in \N, i \leq j \}$
converges to some $A \in \calE$ with $\mu^+(A) = L$.
\end{theorem}

Here and in what follows, the notation $d_{\mu^+}$ is used to represent both a pseudometric on $\calF$ and its restriction to $\calE$. 

Theorem~\ref{closed_limsup_liminf} follows immediately from Theorem~\ref{complete_limsup_liminf}, since a subspace of $\calD$ is complete if and only if its image under the quotient map is complete. Theorem~\ref{closed_limsup_liminf} is framed in terms of the quotient space $\calD / d$ because only in a Hausdorff space is a complete subspace necessarily closed.

The proof of Theorem~\ref{complete_limsup_liminf} requires two propositions.

\begin{proposition} \label{countable_complete}
Consider a set $X$, a $\sigma$-field $\calF \subseteq \calP(X)$, and a countably subadditive set function $\mu : \calF \rightarrow [0, 1]$. The pseudometric space $(\calF, d_{\mu})$ is complete.
\end{proposition}

\begin{proof}
Consider a Cauchy sequence $( A_i )_{i \in \N}$ in $\calF$. Without loss of generality, it may be assumed $\mu(A_i \triangle A_j) < 2^{-i}$ for $i \leq j$, since one can always choose a subsequence with this property, and a Cauchy sequence has the same limits as any convergent subsequence. Define $A := \cup_{i \in \N} \cap_{j = i}^{\infty} A_j \in \calF$. Then
\[
\mu(A_i \triangle A) \leq \sum_{j = i}^{\infty} \mu(A_j \triangle A_{j + 1}) < 2^{1-i},
\]
for each $i$, implying $A_i \rightarrow A$. Hence $(\calF, d_{\mu})$ is complete.
\end{proof}

\begin{proposition} \label{uniform_continuous}
Consider a set $X$, a field $\calF \subseteq \calP(X)$, a subadditive set function $\mu^+ : \calF \rightarrow [0, 1]$, and a set function $\mu^- : \calF \rightarrow [0, 1]$ that is co-subadditive with respect to $\mu^+$. Then $\mu^+$ and $\mu^-$ are uniformly continuous with respect to $d_{\mu^+}$.
\end{proposition}

\begin{proof}
It follows from subadditivity of $\mu^+$ that
\begin{equation}
\lvert \mu^+(A) - \mu^+(B) \rvert \leq \mu^+(A \triangle B)
\end{equation}
and from cosubadditivity of $\mu^-$ that
\begin{equation}
\lvert \mu^-(A) - \mu^-(B) \rvert \leq \mu^+(A \triangle B)
\end{equation}
for any $A, B \in \calF$. This implies uniform continuity of $\mu^+$ and $\mu^-$.
\end{proof}

Theorem~\ref{complete_limsup_liminf} can now be proved. First suppose $(\calE, d_{\mu^+})$ is complete, and consider $( A_i )_{i \in \N}$ in $\calE$ such that 
\[
\{ \mu^+(B_{ij}) : i, j \in \N, i \leq j \}
\mbox{ and }
\{ \mu^-(B_{ij}) : i, j \in \N, i \leq j \}
\]
converge to a common limit $L$, where $B_{ij} := \cup_{k = i}^j A_k$ for each $i$ and $j$. Given $\epsilon > 0$, there is $N \in \N$ such that $\lvert \mu^+(B_{ij}) - L \rvert < \epsilon$ and $\lvert \mu^-(B_{ij}) - L \rvert < \epsilon$ for all $i, j \in \N$ with $N \leq i \leq j$. Consider $i, j , k, l \in \N$ with $N \leq i \leq j$ and $N \leq k \leq l$. Define $m := \min \{ i, k \}$ and $M := \max \{ j, l \}$, and note $B_{ij} \subseteq B_{mM}$ and $B_{kl} \subseteq B_{mM}$. Then
\begin{eqnarray*}
\mu^+(B_{ij} \triangle B_{kl}) &\leq& \mu^+(B_{mM} \triangle B_{ij}) + \mu^+(B_{mM} \triangle B_{kl}) \\
&=& \mu^+(B_{mM} \setminus B_{ij}) + \mu^+(B_{mM} \setminus B_{kl}) \\
&\leq& \mu^+(B_{mM}) - \mu^-(B_{ij}) + \mu^+(B_{mM}) - \mu^-(B_{kl}) \\
&\leq& 2 \lvert \mu^+(B_{mM}) - L \rvert + \lvert \mu^-(B_{ij}) - L \rvert + \lvert \mu^-(B_{kl}) - L \rvert \\
&<& 4 \epsilon,
\end{eqnarray*}
where the third line uses the co-superadditivity of $\mu^+$. Hence $\{ B_{ij} : i, j \in \N, i \leq j \}$ is a Cauchy net containing the Cauchy sequence $( A_i )_{i \in \N}$, and both sequence and net converge to some $A \in \calE$. Moreover,
\[
\lvert \mu^+(A) - L \rvert \leq \lvert \mu^+(A_i) - \mu^+(A) \rvert + \lvert \mu^+(A_i) - L \rvert
\]
for any $i \in \N$. But $i$ can be chosen so that the summands on the right-hand side of the preceding inequality are arbitrarily small, since $\lvert \mu^+(A_i) - \mu^+(A) \rvert \leq \mu^+(A_i \triangle A)$. Hence $\mu^+(A) = L$.

For the converse, let $S$ be the Stone space of the Boolean algebra $\calF$ and let $\phi : \calF \rightarrow \calC$ be the canonical Boolean isomorphism mapping $\calF$ onto the field $\calC \subseteq \calP(S)$ comprised of the clopen (ie. simultaneously open and closed) subsets of $S$. (It is not necessary to be familiar with Stone spaces in order to understand what follows: it is sufficient to know that such an isomorphism exists for some compact topological space $S$.) The proof involves first defining a countably subadditive set function $\hat{\mu}^+ :\calP(S) \rightarrow [0,1]$ and a set function $\hat{\mu}^- : \overline{\calC} \rightarrow [0,1]$, where $\overline{\calC}$ is the closure of $\calC$ with respect to $d_{\hat{\mu}^+}$. The essential properties of these set functions are:
\begin{enumerate}
\item $\mu^+ = \hat{\mu}^+ \circ \phi$ and $\mu^- = \hat{\mu}^- \circ \phi$;

\item $\hat{\mu}^+$ agrees with $\hat{\mu}^-$ on $\overline{\phi(\calE)}$ (where $\phi(\calE)$ is the image of $\calE$ under $\phi$); and

\item $(\overline{\phi(\calE)}, d_{\hat{\mu}^+})$ is complete.
\end{enumerate}

For each $B \in \calP(S)$, define
\[
\hat{\mu}^+(B) := \inf \sum_{i \in \N} \mu^+(A_i),
\]
where the infimum is over sequences $( A_i )_{i \in \N}$ in $\calF$ such that $B \subseteq \cup_{i \in \N} \phi(A_i)$. (This infimum always exists because $B \subseteq \phi(X) = S$.) For $A \in \calF$, it is immediate from the definition that $\hat{\mu}^+(\phi(A)) \leq \mu^+(A)$. Moreover, any sequence $( A_i )_{i \in \N}$ in $\calF$ with $\phi(A) \subseteq \cup_{i \in \N} \phi(A_i)$ contains a finite subsequence $( A_{i_k} )_{k = 1}^n$ with $\phi(A) \subseteq \cup_{k = 1}^n \phi(A_{i_k})$, since $S$ is compact and elements of $\calC$ are both open and closed. Then $A \subseteq \cup_{k = 1}^n A_{i_k}$ and
\[
\mu^+(A) \leq \sum_{k=1}^n \mu^+(A_{i_k}) \leq  \sum_{i \in \N} \mu^+(A_i).
\]
Hence, $\mu^+(A) \leq \hat{\mu}^+(\phi(A))$ and $\mu^+ = \hat{\mu}^+ \circ \phi$.

Note $d_{\hat{\mu}^+}(A,B) = d_{\mu^+}(\phi^{-1}(A),\phi^{-1}(B))$ for $A, B \in \calC$. It follows that $\mu^- \circ \phi^{-1}$ is a uniformly continuous map from $\calC$ to $[0, 1]$, using Proposition~\ref{uniform_continuous}. Thus $\mu^- \circ \phi^{-1}$ has a uniformly continuous extension $\hat{\mu}^- : \overline{\calC} \rightarrow [0, 1]$ (see Proposition~12.13 of~\cite{james1987}). Then $\mu^- = \hat{\mu}^- \circ \phi$. Moreover, $\hat{\mu}^+(B) = \hat{\mu}^-(B)$ for any $B \in \overline{\phi(\calE)}$, since this is true for any $B \in \phi(\calE)$, and the coincidence set of $\hat{\mu}^+$ and $\hat{\mu}^-$ must be closed (see Proposition~6.9 of~\cite{james1987}).

Also note $\hat{\mu}^+$ is countably sub-additive, by an argument used in the proof of Lemma~III.5.5 of~\cite{dunford1958}. Thus $(\calP(S), d_{\hat{\mu}^+})$ is complete by Proposition~\ref{countable_complete}, hence $(\overline{\phi(\calE)}, d_{\hat{\mu}^+})$ is also complete.

Now consider a Cauchy sequence $( A_n )_{n \in \N} \subseteq \calE$. The sequence $( \phi(A_n) )_{n \in \N}$ is Cauchy in $\phi(\calE)$, and thus converges to some $B \in \overline{\phi(\calE)}$. Choose a subsequence $( \phi(A_{n_i}) )_{i \in \N}$ such that $\hat{\mu}^+(\phi(A_{n_i}) \triangle B) < 2^{-i}$ for all $i$. Then for any $i \leq j$,
\[
\hat{\mu}^+((\cup_{k = i}^j \phi(A_{n_k})) \triangle B) \leq \sum_{k = i}^j \hat{\mu}^+(\phi(A_{n_k}) \triangle B) < 2^{1-i}.
\]
Using the inequalities (3.1) and (3.2), it follows that the nets 
\[
\{ \hat{\mu}^+(\cup_{k = i}^j \phi(A_{n_k})) : i, j \in \N, i \leq j \}
\mbox{ and }
\{ \hat{\mu}^-(\phi(\cup_{k = i}^j A_{n_k})) : i, j \in \N, i \leq j \}
\]
converge to a common limit $L := \hat{\mu}^+(B) = \hat{\mu}^-(B)$. Since $\mu^+ = \hat{\mu}^+ \circ \phi$ and $\mu^- = \hat{\mu}^- \circ \phi$, these nets are precisely
\[
\{ \mu^+(\cup_{k = i}^j A_{n_k}) : i, j \in \N, i \leq j \}
\mbox{ and }
\{ \mu^-(\cup_{k = i}^j A_{n_k}) : i, j \in \N, i \leq j \}.
\]
By assumption, the net
\[
\{ \cup_{k = i}^j A_{n_k} : i, j \in \N, i \leq j \}
\]
converges to some $A \in \calE$, and in particular the sequence $( A_{n_i} )_{i \in \N}$ converges to~$A$. The original Cauchy sequence $( A_n )_{n \in \N}$ must therefore converge to $A$, implying $(\calE, d_{\mu^+})$ is complete. This completes the proof.

Note that one can replace the unions with intersections in Theorem~\ref{complete_limsup_liminf}, and the result remains true, with a similar proof. Thus Theorem~\ref{closed_limsup_liminf} also remains true if the least upper bounds are replaced by greatest lower bounds.

\section{Applications}

A reviewer of this paper invited speculation regarding potential applications of the above results ``in number theory or other parts of mathematics''. 
The author is currently investigating properties of the monotone integral induced by upper asymptotic density, defined in the manner of Greco~\cite{greco1977}, and in particular the relationship between that integral and the {\em Ces\`aro mean}, defined as 
\[
\lim_{n \rightarrow \infty} \frac{a_1 + \cdots + a_n}{n}
\]
for certain sequences of real numbers $(a_i)_{i\in \N}$. The author hopes the above results will facilitate that investigation, and subsequently lead to insights in the theory of ergodic processes, where Ces\`aro means play an important role.

Number theorists are interested in whether certain classes of natural numbers have asymptotic density. For example, the class of {\em abundant numbers}, consisting of natural numbers $n$ for which the sum of divisors is at least $2n$, is known to have asymptotic density (see Theorem~7.16 of~\cite{nathanson2000}), and moreover that asymptotic density has been estimated~\cite{deleglise1998}. Conceivably, Theorem~\ref{calF_complete} might facilitate proof that a given class of natural numbers has asymptotic density, by first expressing it as the limit of a sequence of classes with asymptotic density. Theorem~\ref{closed_limsup_liminf} may provide a method of estimating the asymptotic density of such a class, as the limit of a net of the form described in that theorem.

\section*{Acknowledgment}
The author is grateful to K.~P.~S.~Bhaskara Rao for helpful conversations regarding this manuscript.

\bibliographystyle{amsplain}
\bibliography{Keith_refs}

\providecommand{\bysame}{\leavevmode\hbox to3em{\hrulefill}\thinspace}
\providecommand{\MR}{\relax\ifhmode\unskip\space\fi MR }
\providecommand{\MRhref}[2]{%
  \href{http://www.ams.org/mathscinet-getitem?mr=#1}{#2}
}
\providecommand{\href}[2]{#2}
\begin{thebibliography}{10}

\bibitem{blass2001}
A.~Blass, R.~Frankiewicz, G.~Plebanek, and C.~Ryll-Nardzewski, \emph{A note on extensions of asymptotic density}, Proc. Amer. Math. Soc. \textbf{129} (2001), 3313--3320.

\bibitem{blumlinger1996}
M.~Bl\"umlinger, \emph{L\'evy group action and invariant measures on $\beta \mathbb{N}$}, Trans. Amer. Math. Soc. \textbf{348} (1996), 5087--5111.

\bibitem{buck1946}
R.~C. Buck, \emph{The measure theoretic approach to density}, Amer. J. Math. \textbf{68} (1946), no.~4, 560--580.

\bibitem{buck1953}
\bysame, \emph{Generalized asymptotic density}, Amer. J. Math. \textbf{75} (1953), no.~2, 335--346.

\bibitem{deleglise1998}
M.~Del\'eglise, \emph{Bounds for the density of abundant integers}, Exp. Math. \textbf{7} (1998), no.~2, 137--143.

\bibitem{dunford1958}
N.~Dunford and J.~T. Schwartz, \emph{Linear operators part {I}: General theory}, Pure and Applied Mathematics, vol.~7, Wiley-Interscience, 1958.

\bibitem{gangopadhyay1999}
S.~Gangopadhyay and B.~V. Rao, \emph{Completeness of ${L}_1$ spaces over finitely additive probabilities}, Colloq. Math. \textbf{80} (1999), no.~1, 83--95.

\bibitem{greco1977}
G.~H. Greco, \emph{Integrale monotono}, Rend. Semin. Mat. U. Pad. \textbf{57} (1977), 149--166.

\bibitem{james1987}
I.~M. James, \emph{Topological and uniform spaces}, Undergraduate Texts in Mathematics, Springer-Verlag, 1987.

\bibitem{kadane1995}
J.~B. Kadane and A.~O’Hagan, \emph{Using finitely additive probability: uniform distributions on the natural numbers}, J. Amer. Statist. Assoc. \textbf{90} (1995), 626--631.

\bibitem{kunisada2017}
R.~Kunisada, \emph{Density measures and additive property}, J. Number Theory \textbf{176} (2017), 184--203.

\bibitem{kunisada2019}
\bysame, \emph{On a relation between density measures and a certain flow}, Proc. Amer. Math. Soc. \textbf{147} (2019), no.~5, 1941--1951.

\bibitem{kunisada2022}
\bysame, \emph{On the additive property of finitely additive measures}, J. Theoret. Probab. \textbf{35} (2022), 1782--1794.

\bibitem{letavaj2015}
P.~Letavaj, L.~Mi{\v s}\'ik, and M.~Sleziak, \emph{Extreme points of the set of density measures}, J. Math. Anal. Appl. \textbf{423} (2015), 1150--1165.

\bibitem{maharam1976}
D.~Maharam, \emph{Finitely additive measures on the integers}, Sankhya A \textbf{38} (1976), no.~1, 44--59.

\bibitem{mekler1984}
A.~H. Mekler, \emph{Finitely additive measures on {N} and the additive property}, Proc. Amer. Math. Soc. \textbf{92} (1984), no.~3, 439--444.

\bibitem{nathanson2000}
M.~B. Nathanson, \emph{Elementary methods in number theory}, Graduate Texts in Mathematics, vol. 195, Springer-Verlag, 2000.

\bibitem{polya1929}
G.~P\'olya, \emph{Untersuchungen \"uber l\"ucken und singularit\"aten von potenzreihen}, Math. Z. \textbf{29} (1929), 549--640.

\bibitem{sleziak2008}
M.~Sleziak and M.~Ziman, \emph{L\'evy group and density measures}, J. Number Theory \textbf{128} (2008), 3005--3012.

\bibitem{vandouwen1992}
E.~K. van Douwen, \emph{Finitely additive measures on $\mathbb{N}$}, Topology Appl. \textbf{47} (1992), 223--268.

\end{thebibliography}

\end{document}